\documentclass[11pt,reqno]{amsart}
\newcommand{\mysection}[1]{\section{#1}
      \setcounter{equation}{0}}
      \usepackage{amsmath}

\newtheorem{theorem}{Theorem}[section]
\newtheorem{lemma}[theorem]{Lemma}

\newtheorem{corollary}[theorem]{Corollary}

\theoremstyle{definition}
\newtheorem{assumption}{Assumption}[section]

\newtheorem{example}{Example}[section]

\theoremstyle{remark}
\newtheorem{remark}{Remark}[section]

\newcommand\bG{\mathbb{G}}
\newcommand\bR{\mathbb{R}}

\newcommand\bW{\mathbb{W}}

\newcommand\frm{\mathfrak{m}}

\newcommand\cB{\mathcal{B}}

\newcommand\cF{\mathcal{F}}

\newcommand\cK{\mathcal{K}}
\newcommand\cL{\mathcal{L}}
\newcommand\cM{\mathcal{M}}

\newcommand\cO{\mathcal{O}}
\newcommand\cP{\mathcal{P}}

\newcommand\cR{\mathcal{R}}

\newcommand{\ga}{\mathfrak{a}}
\newcommand{\gb}{\mathfrak{b}}

\begin{document}

\title[Accelerated schemes]{
 Accelerated finite difference 
schemes for   linear  stochastic partial differential 
equations in the whole space}

\author[I. Gy\"ongy]{Istv\'an Gy\"ongy}
\address{School of Mathematics and Maxwell Institute,
University of Edinburgh,
King's  Buildings,
Edinburgh, EH9 3JZ, United Kingdom}
\email{gyongy@maths.ed.ac.uk}

\author[N. Krylov]{Nicolai Krylov}%
\thanks{The work of the second author was partially supported
by NSF grant DMS-0653121}
\address{127 Vincent Hall, University of Minnesota,
Minneapolis,
       MN, 55455, USA}
\email{krylov@math.umn.edu}


\begin{abstract}
We give sufficient conditions under which the 
convergence of finite difference approximations 
in the space variable of  the solution to the Cauchy 
problem for linear  stochastic PDEs of parabolic type 
can be accelerated 
to any given order of convergence by Richardson's method.  
\end{abstract}

\maketitle

\mysection{Introduction}                   \label{section02.04.06}

Stochastic partial differential equations (SPDEs) play important roles 
in many applied fields. Here we consider linear second order nondegenerate 
parabolic SPDEs. These equations arise, for example, 
in nonlinear filtering of partially observable
diffusion processes. There are various methods developed 
in the literature to solve them numerically. 
In this paper we apply the method of finite differences in the space 
variable, while the time
variable changes continuously.  It is known that in general
the error  of the finite difference approximations 
in the space variable is proportional to 
the parameter $h$ of the finite difference, see,  
e.g., \cite{Y} or  \cite{Y1}. 
Our aim is to show that the convergence 
of these approximations can be 
accelerated by an implementation of Richardson's idea to 
SPDEs. We prove that for linear parabolic stochastic PDEs 
driven by Wiener processes the finite difference approximations 
$u^h$ admit power series expansions in the parameter $h$. 
This is Theorem \ref{theorem 5.25.11}, 
one of the main results of the paper. 
Hence we get Theorem 
\ref{theorem 5.25.3}, our first result on acceleration of 
finite difference schemes for SPDEs. 
It says that if the coefficients 
and the data are sufficiently regular 
then the convergence of finite difference approximations 
can be accelerated to any high order 
by taking appropriate mixture of approximations with different 
step sizes. In the special case 
of symmetric finite difference schemes,  
Example \ref{example 5.22.2} below, 
the coefficients of odd powers in the expansions vanish. 
Hence it follows, see Theorem \ref{theorem 5.8.5.9}, 
that the error of symmetric finite difference schemes 
is proportional to $h^2$ without acceleration, 
and we can accelerate more effectively.

The SPDEs we consider in this paper 
are given in $[0,T]\times\bR^d$. 
The finite difference schemes are given in
$[0,T]\times\mathbb G_h$, where 
$\mathbb G_h$ are grids in the space 
variable. 
The supremum in $t\in[0,T]$ and $x\in\mathbb G_h$ 
of the remainder terms and of the approximation errors 
in the expansions in Theorem  \ref{theorem 5.25.11}
and Theorems \ref{theorem 5.25.3}-\ref{theorem 5.8.5.9}, 
respectively  are estimated.  
To prove these results we consider the finite difference schemes 
given not only on the grids, but on the whole $\mathbb R^d$, and 
obtain a more general theorem, Theorem \ref{theorem 5.29.1},  
that establishes a power expansion in $h$ for the $L_2$-solutions of the 
schemes on $[0,T]\times\bR^d$, with the 
remainder estimated in terms of Sobolev norms 
in the whole $\mathbb R^d$. Hence we estimate 
the sup norm and also 
discrete Sobolev norms of the remainder by Sobolev's 
embedding theorems, and get 
our theorems on accelerated finite difference schemes, formulated 
in terms of supremum norm and also in discrete Sobolev norms 
of functions over $\mathbb G_h$.

In the special case when the stochastic terms in the 
equations vanish the above mentioned 
theorems are results 
on accelerated finite  difference schemes for deterministic PDEs.  Similar
results  on {\em monotone\/} finite difference schemes  
for parabolic and elliptic PDEs, which may degenerate,  
are proved in \cite{GK3} on the basis of derivative 
estimates on the supremum norm obtained 
in \cite{GK1}-\cite{GK2} for solutions to monotone finite 
difference schemes. The finite difference
schemes in the present article are not necessarily
monotone. 

The idea of accelerating the 
convergence of finite difference approximations to 
deterministic PDEs by suitable mixtures of approximations 
with different step-sizes is  due to L.F.~Richardson, 
see \cite{Ri} and \cite{RG}. This method is often called 
{\em Richardson's method\/} or {\em extrapolation to the 
limit\/}, and is applied to various types of approximations. 
It is used in \cite{GK4}-\cite{GK5}  
to accelerate splitting up approximations for a large class 
of deterministic evolution equations, 
including second order parabolic equations  
and symmetric hyperbolic system of first order PDEs. 
Richardson's idea is implemented to the law of Euler's 
approximations for stochastic differential equations 
in \cite{TT}, \cite{BT} and \cite{MT}. 
There is a lot of other applications of Richardson's 
method. 
The reader is referred to the survey papers 
\cite{Br} and \cite{J} for a review  
on  the method, and  to textbooks 
(for instance,  \cite{Ma} and \cite{MS})
concerning finite difference 
methods and their accelerations. 
 We note that previous extrapolation results  
for stochastic equations, i.e. in \cite{TT}, and in its generalizations 
\cite{MT} and \cite{KPH},  
are concerned with {\it week 
approximations} of stochastic differential equations. In contrast 
our main results are error expansions for {\it strong 
convergence} of finite difference approximations in the space variable 
for stochastic parabolic equations, and as far as we know these are the first 
results in this direction. 

In light of the results of the present paper it is
natural to look  for accelerated space {\em and time}
 discretized
schemes, say by using time discretization
to solve the systems of ordinary stochastic equations which we
obtain after discretizing the space. However, one knows 
 that if the values of the driving multidimensional Wiener process 
 are available only at the grid points, then in general one cannot construct a scheme 
 with (strong) rate of convergence   better than $\sqrt{\tau}$, where 
 $\tau$ is the mesh-size of the time grid. On the other hand, in some particular 
 cases, e.g., when the Wiener process is one-dimenional, or some special data, 
 like iterated stochastic integrals of the components of the Wiener 
 processes are available, then one can have accelerated fully discretised numerical 
 schemes for SPDEs. (See, e.g., \cite{KP}  for high order strong 
 approximations of stochastic differential equations when appropriate iterated 
 stochastic integrals of the Wiener processes are used in 
the numerical schemes.)

%Our paper seems to be the first one to justify the method 
%for finite difference approximations in the space variable 
%for stochastic parabolic equations. 

We did not try to make our results as sharp or as
general as possible. The main goal of the article
is to show a method of approximating. 
We plan to extend 
our results to the case of degenerate parabolic SPDEs 
in the continuation of this paper.

In conclusion we introduce some notation used everywhere
below.
    Throughout the paper $\bR^{d}$ is a Euclidean space
of points $x=(x^{1},...,x^{d})$, 
 and $T>0$ is a fixed finite constant. 
 We set
$$
D_{i}=\partial/\partial x^{i},\quad
  i=1,\dots,d.
$$
Also let $D_{0}$ be the unit operator.

Let $(\Omega,\cF,P)$ be a complete probability space,
$\cF_{t}$, $t\geq0$, be an increasing filtration
of sub $\sigma$-fields of $\cF$, such that $\cF_{0}$
is complete with respect to $(\cF,P)$. By
$\cP$ we denote the $\sigma$-field of predictable subsets of
$\Omega\times [0,\infty)$ generated by $\cF_{t}$, and 
$\cB(\bR^d)$ is the $\sigma$-algebra of the Borel subsets 
of $\bR^d$.  
We assume that on $\Omega$ we are given
a sequence of $\mathcal F_t$-adapted 
independent Wiener processes $\{w^{\rho}\}_{\rho=1}^{\infty}$ 
such that for every integer $\rho\geq1$ and for all $0\leq s\leq t$ 
the increments  $w_t^{\rho}-w_s^{\rho}$  
are independent of $\mathcal F_s$. 
Unless otherwise stated throughout the paper we use the summation
convention over repeated integer valued indices. 
  For functions 
$u=u(\omega,t,x)$, $\omega\in \Omega$, $t\in[0,T]$,  $x\in\bR^d$, we 
use the notation $D^{l}u=D^{l}u(x)$ for  the collection of 
$l$th order derivatives of $u$ with respect
to $x$ and $|D^{l}u|^{2}=|D^{l}u(x)|^{2}$ is the sum of squares of
all $l$th order derivatives at $x$. If $u$ is an $l_2$-valued function 
then the differentiability of it is understood in the sense of $l_2$-valued 
functions and $|D^{l}u(x)|^{2}_{l_2}$ means the the sum of squares of the 
$l_2$-norm of  all $l$th order derivatives at $x$. 
 For basic notions and notation concerning the theory 
of linear stochastic partial differential equations we refer to 
\cite{R}.  

\bigskip
 
{\bf Acknowledgments.} The first version of this paper was 
presented at the ICMS conference on `Numerical Analysis of Stochastic PDEs' 
(Edinburgh,  May 2009), 
organised by Evelyn Buckwar and Gabriel Lord, and at the `Stochastic Analysis' session 
of the 7th International ISAAC Congress 
(Imperial College, London, August 2009)
organized by Dan Crisan and Terence Lyons.  
We thank the organisers for the invitations. 

 We are sincerely grateful to the referees for their careful work which
helped improve the presentation of the paper.

\mysection{Formulation of the main results}

We consider the equation

\begin{equation}                          \label{3.25.4.9}
du_{t}=(\cL_{t}u_{t}
+f _{t})\,dt
+(\cM^{\rho}_{t}u_{t}+g^{\rho}_{t})\,dw^\rho_{t}, 
\end{equation}
for $(t,x)\in[0,T]\times\bR^d=: H_T$ 
with some initial condition where 
$$
\cL_{t}\phi= a^{\alpha\beta}_{t}D_{\alpha}
D_{\beta} \phi ,
\quad
\cM^{\rho}_{t}\phi= b^{\alpha \rho}_{t}D_{\alpha}
  \phi,
$$
and  $\{w^{\rho}\}_{\rho=1}^{\infty}$ is a sequence of 
independent Wiener processes given on a probability space 
$(\Omega, \cF,P)$ equipped with a filtration $(\cF)_{t\geq0}$ 
such that $w^{\rho}_t$ is $\mathcal F_t$-measurable and 
$w^{\rho}_t-w^{\rho}_s$ is independent of $\mathcal F_s$ 
for all $0\leq s\leq t$ and integers $\rho\geq1$. 
 Here and below the summation
with respect to $\alpha$ and $\beta$ is performed over the set 
$ \{0,1,...,d\}$ and with respect to $\rho$ in the range
$\{1,2,...\}$. Assume that,
for $\alpha,\beta\in \{0,1,...,d\}$, we have $a^{\alpha\beta}_{t}=
 a^{\beta\alpha}_{t}$ and
$a^{\alpha\beta}_{t}=a^{\alpha\beta}_{t}(x)$
 are real-valued and $b^{\alpha}_{t}
=(b^{\alpha \rho}_{t})_{\rho=1}^{\infty}$   are $l_{2}$-valued
$\cP\times\cB(\bR^d)$-measurable functions on $\Omega\times H_{T}$.

Let $m\geq1$ be an integer and let $W^{m}_{2}$ be  
 the  usual
Hilbert-Sobolev space 
  of functions on $\bR^d$  with norm $\|\cdot\|_{W^{m}_{2}}$.

\begin{assumption} 
                                   \label{assumption 5.19.2}
(i) For each
$(\omega,t)$ the functions $a^{\alpha\beta}_{t}$
 are $m$ times and the functions $b^{\alpha}_{t}$ are $m+1$
times continuously
differentiable in $x$. There 
  exist  constants $K_{l}$, $l=0,...,m+1$, such that for all values
of indices and arguments  we have
$$
|D^{l}a^{\alpha\beta}_{t}|\leq K_{l},\quad l\leq  m,\quad
|D^{l}b^{\alpha}_{t}|_{l_{2}}\leq K_{l},\quad l\leq m+1.
$$

(ii) There is a constant $\kappa>0$ such that 
for all $(\omega,t,x)\in\Omega\times H_T$ and $z\in\bR^{d}$
$$
\sum_{i,j=1}^{d}(2a^{ij}_{t}-b^{i\rho}_{t} b^{j\rho}_{t})z^iz^j
\geq\kappa |z|^2.
$$

\end{assumption}

\begin{assumption}
                                   \label{assumption 5.19.1}
We have $u_{0}\in L_{2}(\Omega,\cF_{0},W^{m+1}_{2})$.
The function   $f_{t}$ 
 is  $W^{m}_{2}$-valued, $g^{\rho}_{t}$,
$ \rho =1,2,...$, are $W^{m+1}_{2}$-valued  
functions 
given on $\Omega\times[0,T]$ and they
 are  predictable. Moreover,
for $ g_{t}
:=(g^{\rho}_{t} )_{\rho=1}^{\infty }$ and 
$$
\|g_{t}\|^{2}_{W^{l}_{2}}:=
\sum_{\rho=1}^{\infty}\|g^{\rho}_{t}\|^{2}_{W^{l}_{2}}
$$
we have
$$
E\int_{0}^{T}(\|f_{t}\|^{2}_{W^{m}_{2}}
+\|g_{t}\|^{2}_{W^{m+1}_{2}})\,dt+E\|u_{0}\|^{2}_{W^{m+1}_{2}}
=:\cK^{2}_{m}<\infty.
$$

\end{assumption}
 
\begin{remark}                                     \label{remark 1.27.10.9}
If Assumption \ref{assumption 5.19.1} 
holds with $m>d/2$, 
then by Sobolev's embedding of $W^m_2$ into $C_b$, 
the space of bounded continuous functions, for almost all 
$\omega$ we can find a continuous function of $x$ which 
equals to $u_0$ almost everywhere. Furthermore, 
for each $t$ and $\omega$ we have 
continuous functions of $x$ which coincide with  $f_t$ and $g_t$, 
for almost every $x\in\bR^d$. Therefore when 
Assumption \ref{assumption 5.19.1} holds 
with $m>d/2$, we always assume that $u_0$, $f_t$ and $g_t$ 
are continuous in $x$ for all $t$. 
\end{remark} 
 
The solutions of \eqref{3.25.4.9} will be looked for
in the Hilbert space 
$$
\bW^{m+2}_{2}(T)=L_{2}(\Omega
\times[0,T],\cP,W^{m+2}_{2}).
$$
 One knows,   see e.g., \cite{KR} or \cite{R}, how to define
stochastic integrals of
Hilbert-space valued processes and equation \eqref{3.25.4.9} 
is understood accordingly. Observe that since $u_{0}
\in L_{2}(\Omega,\cF_{0},W^{m}_{2})$ the solutions
of \eqref{3.25.4.9} automatically  are continuous
$W^{m}_{2}$-valued processes (a.s.).
  
We are going to use the following classical result (see,
for instance, 
 Theorem 5.1, Remark 5.6, and Theorem 7.1 of \cite{Kr99}).

\begin{theorem}
                                            \label{theorem 5.25.1}
Under the above assumptions
 there exists a unique solution $u\in\bW^{m+2}_{2}(T)$
of \eqref{3.25.4.9} with initial condition $u_{0}$.
Furthermore, with probability one the function
$u_{t}$ is a continuous $W^{m+1}_{2}$-valued function
and there exists a constant $N$ depending only on $T$, $d,\kappa$,
$m$, and $K_{l},l\leq m+1$, such that
$$
E\sup_{t\leq T}\|u_{t}\|^{2}_{W^{m+1}_{2}}
+E\int_{0}^{T}\|u_{t}\|^{2}_{W^{m+2}_{2}}\,dt
\leq N\cK^{ 2}_{m}.
$$
\end{theorem}
\begin{remark}
                                            \label{remark 5.25.1}
In the future we are going to assume that
$m+1>d/2$. Then by Sobolev embedding theorems the solution
$u_{t}(x)$ from Theorem \ref{theorem 5.25.1} is a continuous
function of $(t,x)$ (a.s). More precisely, with probability
one, for any $t$ one can find a continuous function
of $x$ which equals $u_{t}(x)$ for almost all $x$
and, in addition, the so constructed modification is continuous with
respect to the couple $(t,x)$.
 \end{remark}

We are interested in approximating the solution
by means of solving a semidiscretized version of
\eqref{3.25.4.9} when partial derivatives are replaced
with finite differences. 
For $\lambda=0$ set $ \delta_{h,\lambda} $ to be
the unit operator and for the other values of $\lambda
\in\bR^{d}$ let
$$
 \delta_{h,\lambda}u(x)=\frac{u(x+h\lambda)
-u(x)}{h}\quad
 \text{for $h\in\bR\setminus\{0\}$} .
$$
We draw the reader's attention to the fact that
$h$ can be of any sign. This will be important in the future.

To introduce difference equations
we take a finite set $\Lambda\subset\bR^{d}$
containing the origin, and consider the equation
\begin{equation}                             \label{5.20.3}
du^{h}_{t}=(
L^{h}_{t}u^{h}_{t}
+f_{t} )\,dt
+(M^{h,\rho}_{t}u^{h}_{t}
+g^{\rho}_{t})\,dw^{\rho}_{t},
\end{equation}
with  
$$
L^{h}_{t}\phi =
\ga^{\lambda\mu}_{t}\delta_{h,\lambda}\delta_{-h,\mu}\phi,
\quad
M^{h,\rho}_{t}\phi=\gb^{\lambda \rho}_{t}\delta_{h,\lambda}
\phi,
$$
where the summation is performed over
$\lambda,\mu\in\Lambda$ and in \eqref{5.20.3}
also with respect to $\rho=1,2,\dots$. 
Assume that, for $\lambda,\mu\in \Lambda$,  
$\ga^{\lambda\mu}=\ga^{\lambda\mu}_{t}(x)$
 are real-valued and 
$\gb^{\lambda }
=\gb^{\lambda }_{t}(x)
=(\gb^{\lambda \rho}_{t}(x) )_{\rho=1}^{\infty} $   
are $l_{2}$-valued
functions on $\Omega\times H_{T}$,
 measurable with respect to $\cP\times\cB(\bR^d)$. 

 Set $\Lambda_{0}:=\Lambda\setminus\{0\}$. 
Let  $\frm\geq0$ be an
integer. Set  
$ \bar \frm=\max(\frm,1) $, 
and let $A_0$, $A_1$,..., 
$A_{\bar \frm}$ be some constants. 
The functions $\ga$ and $\gb$ 
are supposed to possess the following properties. 

\begin{assumption}
                                        \label{assumption 5.19.3}

(i) For each
$(\omega,t)$ and $\lambda,\mu\in\Lambda_0$ and $\nu
\in\Lambda$, $\ga^{\lambda\mu}_{t}$  are 
 $\bar \frm$  
 times continuously
differentiable in $x$, 
 $\ga^{0\nu}_{t}$, $\ga^{\nu0}_{t}$ 
are $ \frm $ 
times continuously differentiable in $x$ and 
 $\gb^{\nu}_{t}$   are  
 $ \frm $  
 times continuously
differentiable in $x$ as $l_{2}$-valued functions. 
For all values  
of   arguments  we have 
$$
|D^{j}\ga^{\lambda\mu}_{t}|\leq A_{j},\quad
  \lambda,\mu\in\Lambda_0, 
\quad j\leq 
 \bar \frm , 
$$
$$
 |D^{j}\ga^{\lambda0}_{t}|\leq A_{j},\quad
|D^{j}\ga^{0\lambda}_{t}|\leq A_{j}, 
\quad 
 |D^{j}\gb^{\lambda }_{t}|_{l_{2}}\leq A_{j},
 \quad \lambda\in\Lambda,
 \quad j\leq  \frm .
$$

(ii) For all $(\omega,t,x)\in\Omega\times H_T$ and numbers
$z_{\lambda},\lambda\in\Lambda_{0}$,
we have
$$
\sum_{\lambda,\mu\in\Lambda_{0}}
(2\ga^{\lambda  \mu}_{t}-\gb^{\lambda \rho}_{t} \gb^{\mu \rho}_{t})
z_\lambda z_\mu
\geq\kappa \sum_{\lambda\in\Lambda_{0}}z_{\lambda}^2.
$$
\end{assumption}

Introduce 
$$
\bG_{h}=\{ \lambda_{1}h+... +
\lambda_{n}h:n=1,2,...,
 \lambda_{i}\in\Lambda\cup(-\Lambda)\}
$$
and let $l_{2}(\bG_{h})$ be the set of real-valued
functions $u$
on $\bG_{h}$ such that
$$
|u|_{l_{2}(\bG_{h})}^{2}:=|h|^{d}\sum_{x\in\bG_{h}}|u(x)|^{2}
<\infty.
$$
The notation $l_{2}(\bG_{h})$ will also be used for 
$l_{2}$-valued functions like $g$.
\begin{remark}
                                             \label{remark 5.25.2}

Observe that,
under Assumption \ref{assumption 5.19.3} (i), equation \eqref{5.20.3} is
an ordinary It\^o equation  with Lipschitz continuous coefficients for
$l_{2}(\bG_{h})$-valued processes.  Therefore if,
for instance, (a.s.)
$$
\int_{0}^{T}(|f_{t}|_{l_{2}(\bG_{h})}^{2}+
|g_{t}|_{l_{2}(\bG_{h})}^{2})\,dt<\infty,
$$
 and  
Assumption \ref{assumption 5.19.3} (i) holds  
then equation \eqref{5.20.3} has a unique solution
with continuous trajectories in $l_{2}(\bG_{h})$ 
provided that the initial data $u^{h}_{0}\in l_{2}(\bG_{h})$
(a.s.).
\end{remark}

For equation \eqref{5.20.3} to be consistent with
\eqref{3.25.4.9} we impose the following.

\begin{assumption}
                                        \label{assumption 5.22.1}
For all $i,j=1,...,d$ and $ \rho =1,2,...$
$$
\sum_{\lambda,\mu\in\Lambda_{0}}\ga^{\lambda\mu}_{t}\lambda^{i}
\mu^{j}= a^{ij}_{t},
\quad \sum_{\lambda\in\Lambda_{0}}\ga^{\lambda0}_{t}\lambda^{i}
+\sum_{\mu\in\Lambda_{0}}\ga^{0\mu}_{t}\mu^{i}
=
a^{i0}_{t}+a^{0i}_{t}, \quad\ga^{00}_{t}=a^{00}_{t},
$$
$$
 \sum_{\lambda \in\Lambda_{0}}
\gb^{\lambda \rho}_{t}\lambda^{i}=b^{i\rho}_{t},
\quad \gb^{0\rho}_{t}=b^{0\rho}_{t}.
$$

\end{assumption}

 \begin{remark}
Clearly, if 
$$
a_{t}^{ij}=\sum_{\lambda,\mu\in\Lambda_{0}}\ga^{\lambda\mu}_{t}\lambda^{i}
\mu^{j}, \quad i,j=1,...,d
$$
is an invertible matrix for some  $\omega$, $t$, $x$, then 
$\Lambda_0$ spans the whole $\bR^d$. On the other hand, if 
$\Lambda_0$ spans $\bR^d$, then clearly a constant $\kappa'>0$ 
exists such that 
$$
\sum_{\lambda\in\Lambda_0}|\sum_{i}z^i\lambda^i|^2\geq\kappa'|z|^2,
\quad \text{for all $z=(z^1,...,z^d)\in\bR^d$}, 
$$
and therefore  
Assumptions \ref{assumption 5.19.3} (ii) 
and \ref{assumption 5.22.1} imply Assumption 
\ref{assumption 5.19.2} (ii). It is not hard to see
that  
Assumptions \ref{assumption 5.19.2} (ii) and
\ref{assumption 5.22.1} 
do not imply 
Assumption \ref{assumption 5.19.3} (ii), 
in general, unless $\Lambda_0$
is a basis in 
$\bR^d$. 
\end{remark} 

There are several ways to construct appropriate $\ga$ and $\gb$. 

\begin{example}
                                            \label{example 5.22.1}
The most natural, albeit sometimes
not optimal,  way to choose $\ga$ and $\gb$
is to set $\Lambda =\{e_{0}, e_{1},...,  e_{d}\}$, where $e_{0}=0$
and $e_{i}$ is the $i$th basis vector  in $\bR^d$  
and  let 
$$
\ga^{e_{\alpha} e_{\beta}}_{t}= a^{\alpha\beta}_{t},\quad
\gb^{e_{\alpha}\rho }_{t}=b^{\alpha \rho}_{t},
\quad \alpha,\beta=0,1,...,d.
$$
 Thus, in \eqref{5.20.3} 
 the first order derivatives in 
 \eqref{3.25.4.9} are approximated 
by usual finite differences and
\begin{equation}
                                                              \label{5.26.3}
\sum_{\lambda,\mu\in\Lambda_{0}}\ga^{\lambda\mu}_{t}
\delta_{h,\lambda}\delta_{-h,\mu} u=-
 a^{ij}_{t}\delta_{h,e_{i}}\delta_{h,-e_{j}}u,
\end{equation}
which is a standard finite-difference approximation
 of $a^{ij}_{t}D_{i}D_{j}u$. Also notice that
$$
\sum_{\lambda,\mu\in\Lambda_{0}}
\ga^{\lambda\mu}_{t} 
z_\lambda z_\mu= 
a^{ij}_{t}z_{e_{i}}z_{e_{j}},\quad
\sum_{\lambda\in\Lambda_{0}}
\gb^{\lambda \rho}_{t}z_{\lambda}=b^{i\rho}_{t}
z_{e_{i}}.
$$
It follows that $\ga$ and $\gb$ satisfy 
the above assumptions 
as long as $a$ and
$b$ do.

\end{example}
\begin{example}
                                     \label{example 5.22.2}

The second choice  is to use  symmetric
finite differences to approximate the
first-order derivatives. Namely, we take 
$\Lambda_{0}=\{\pm e_{1},...,\pm e_{d}\}$ and
$$
\ga^{0,\pm e_{i} }_{t}
=\ga^{\pm e_{i},0 }_{t}
=\pm(1/4)(a^{0i }_{t}+a^{i0}_{t}),\quad
\gb^{\pm e_{i},\rho}_{t}=\pm(1/2)b^{i,\rho}_{t},
$$
$$
\ga^{00}_{t}=a^{00}_{t},\quad \gb^{0\rho}_{t}=b^{0\rho}_{t},
$$
so that, for instance,
$$
\sum_{\lambda\in\Lambda_{0}}\gb^{\lambda \rho}_{t}\delta_{h,\lambda}
u(x)=\sum_{i=1}^{d}b^{i\rho}_{t}\frac{u(x+he_{i})-u(x-he_{i})}{2h}.
$$
For $\lambda,\mu\in\Lambda_{0}$ we define $\ga^{\lambda\mu}_{t}$
by
$$
\ga^{\pm e_{i},\pm e_{j}}_{t}
=\tfrac{1}{2}a^{ij}_{t},\quad 
\ga^{\pm e_{i},\mp e_{j}}_{t}=0.
$$
Then Assumption \ref{assumption 5.22.1} is satisfied
and formula \eqref{5.26.3} holds again ($a^{ij}=a^{ji}$).
If Assumption \ref{assumption 5.19.2} (ii) is satisfied, then
for any numbers $z_{\lambda}$
$$
\sum_{\lambda,\mu\in\Lambda_{0}}
(2\ga^{\lambda\mu}_{t}-\gb^{\lambda \rho}_{t}
\gb^{\mu \rho}_{t}) z_\lambda
z_\mu=\sum_{i,j=1}^{d}a^{ij}_{t}z_{e_{i}}z_{e_{j}}
+\sum_{i,j=1}^{d}a^{ij}_{t}z_{-e_{i}}z_{-e_{j}}
$$
$$
-(1/4)\sum_{\rho}\big|\sum_{i=1}^{d}b^{i\rho}_{t}z_{e_{i}}
-\sum_{i=1}^{d}b^{i\rho}_{t}z_{-e_{i}}\big|^{2}
\geq \sum_{i,j=1}^{d}a^{ij}_{t}z_{e_{i}}z_{e_{j}}
-(1/2)\sum_{\rho}\big|\sum_{i=1}^{d}b^{i\rho}_{t}z_{e_{i}} \big|^{2}
$$
$$
+\sum_{i,j=1}^{d}a^{ij}_{t}z_{-e_{i}}z_{-e_{j}}
-(1/2)\sum_{\rho}\big|\sum_{i=1}^{d}b^{i\rho}_{t}z_{-e_{i}}\big|^{2}
$$
$$
\geq \kappa\sum_{i=1}^{d}z_{e_{i}}^{2}+
\kappa\sum_{i=1}^{d}z_{-e_{i}}^{2}=\kappa
\sum_{\lambda\in\Lambda_{0}}z_{\lambda}^{2},
$$
so that Assumption  
 \ref{assumption 5.19.3}  (ii) is also
satisfied.   By comparing Theorems \ref{theorem 5.25.3} 
and \ref{theorem 5.8.5.9}
and also definitions \eqref{11.25.11.08} and \eqref{12.25.11.08}  
for approximations 
$\bar u^h$ and $\tilde u^h$ below, notice that 
the above choice of $\ga$ and $\gb$ is better 
than that of the previous example,  in the sense that for 
$\tilde u^h$ we have fewer terms to calculate than for $\bar u^h$ 
to get the same order of 
accuracy of the approximations.  

\end{example}

Our results revolve about the possibility to prove
the existence of random processes $u^{(j)}_{t}(x)$, $t\in[0,T],
x\in\bR^{d}$, 
  $j=0,\dots,k$, for some integer $k\ge0$ 
such that 
they are independent of $h$, $u^{(0)}$ is the solution of
\eqref{3.25.4.9} with initial value $u_{0}$  and  
almost surely  we have
 
\begin{equation}                                                                                           
                                             \label{1.26.4.9} 
u^h_{t}(x)=\sum_{j=0}^{k}\frac{h^{j}}{j!} u^{(j)}_{t}(x)
+ R ^h_{t}(x)
\end{equation}
  for  $h\ne0$  and for all $t\in[0,T]$ 
and $x\in\mathbb G_h$, 
where
$u^h_{t}$  is the solution to \eqref{5.20.3} with initial data
$u_{0}$ and 
$ R ^h$ is a continuous $l_2(\mathbb G_h)$-valued adapted 
process, such that 
\begin{equation}
                                                  \label{5.25.4}
E\sup_{t\in[0,T]}\sup_{x\in\bG_h}|R^h_{t}( x)|^2\leq 
Nh^{2(k+1)}\cK_m^2 
\end{equation}
with a constant $N$ independent of $h$.

\begin{theorem}
                                       \label{theorem 5.25.11}
 Let Assumptions 
  \ref{assumption 5.19.2}, 
 \ref{assumption 5.19.1}, 
 \ref{assumption 5.19.3} 
 and 
 \ref{assumption 5.22.1}  
 hold  with 
$$
\frm=m>k+1+d/2,
$$
where $k \geq0$ is an integer.  
Then expansion \eqref{1.26.4.9}   and estimate 
\eqref {5.25.4} hold with a constant $N$  
depending only on $\Lambda$, $d$, $m$, $K_0$, ..., $K_{m+1}$, 
$A_{0},...,A_{ m }$,
$\kappa$, and $T$. 
\end{theorem}

\begin{remark}                                                  \label{remark 4.20.10.9}
Actually $u^{h}_{t}(x)$ is defined for all $x\in\bR^{d}$
rather than only on $\bG_{h}$ and, as we will 
see from the proof of Theorem
\ref{theorem 5.25.11}, one can replace $\bG_{h}$ in \eqref{5.25.4}
with $\bR^{d}$. 
\end{remark}

\begin{remark}
Let $\Lambda_0$ be a basis in $\bR^d$ such that 
Assumption  \ref{assumption 5.22.1} holds. Then 
Assumption \ref{assumption 5.19.2} (i) implies 
Assumption \ref{assumption 5.19.3} (i), 
and Assumption \ref{assumption 5.19.2} (ii) implies 
Assumption \ref{assumption 5.19.3} (ii)  with
$\frm=m$. 
Thus  if Assumptions \ref{assumption 5.19.2} and 
 \ref{assumption 5.19.1} hold with 
$$
m>k+1+d/2,  
$$
then 
the conditions  of 
Theorem \ref{theorem 5.25.11} are satisfied. 
\end{remark}

Equality \eqref{1.26.4.9} clearly yields 
$$
\delta_{h,\lambda}u^h_{t}(x)=\sum_{j=0}^{k}\frac{h^{j}}{j!} 
\delta_{h,\lambda}u^{(j)}_{t}(x)
+\delta_{h,\lambda}R^h_{t}(x)
$$
for any $\lambda=(\lambda_1,...,\lambda_n)\in \Lambda^n$ and 
integer $n\geq0$, where $\Lambda^0=\{0\}$ and 
 $$
 \delta_{h,\lambda}:=\delta_{h,\lambda_1}\cdot...\cdot\delta_{h,\lambda_n}. 
 $$
 
 Theorem \ref{theorem 5.25.11} can be generalised 
as follows. 

\begin{theorem}                                  \label{theorem 1.25.10.9}
 Let the conditions 
of Theorem \ref{theorem 5.25.11} hold 
with 
$$
\frm=m>k+n+1+d/2 
$$
for some integers $k\geq0$ and $n\geq0$.  Then expansion 
\eqref{1.26.4.9} holds and
for $\lambda =(\lambda_1,...,\lambda_n) \in\Lambda^n$ 
$$
E\sup_{t\in[0,T]}\sup_{x\in\bG_h}|\delta_{h,\lambda}R^h_{t}( x)|^2
+
E\sup_{t\in[0,T]}\sum_{x\in\bG_h}|\delta_{h,\lambda}R^h_{t}( x)|^{2}
|h|^{d} 
\leq 
Nh^{2(k+1)}\cK_m^2, 
$$
where $N$ depends only on  $\Lambda$, $d$, $m$, 
$K_0$, ..., $K_{m+1}$, $A_{0},...,A_{ m }$,
$\kappa$ and $T$.
 
\end{theorem}

We prove Theorem  \ref{theorem 1.25.10.9} 
in Section  \ref{section proof} 
after some preliminaries presented in Section 
 \ref{section 5.25.1}. 
To discuss  the method of acceleration
we fix an integer $k\geq0$ and 
set  
\begin{equation}                       \label{11.25.11.08}
\bar u^h=\sum_{j=0}^{k}b_ju^{2^{-j}h},  
\end{equation}
where, naturally, $u^{2^{-j}h}$   are 
the solutions to \eqref{5.20.3}, 
with $2^{-j}h$ in place of $h$, 
\begin{equation}                       \label{1.26.11.08}
(b_0,b_1,...,b_k)
:=(1,0,0,...,0)V^{-1}
\end{equation}
and $V^{-1}$ is the inverse of the 
Vandermonde matrix with entries
$$
V^{ij}:=2^{-(i-1)(j-1)}, \quad i,j=1,...,k+1.
$$

The following 
consequence 
of Theorem \ref{theorem 5.25.11} is the first main result 
of the paper 
on accelerated convergence.  
Its generalisation 
is presented in Section \ref{section proof}.

\begin{theorem}                     \label{theorem 5.25.3} 

Under the assumptions of Theorem \ref{theorem 5.25.11}
we have
\begin{equation}                                     \label{2.20.10.9}
E\sup_{t\leq T} \sup _{x\in\bG_h}|\bar u^h_{t}( x)-
u^{(0)}_{t}( x)|^{2}
\leq N |h|^{2(k+1)}\cK_m^{2} , 
\end{equation} 
where $N$ depends only on $\Lambda$, $d$, 
$m$, $K_0$, ..., $K_{m+1}$, 
$\kappa$, 
$ A_{0},...,A_{ m } $, and $T$.
 
\end{theorem}

\begin{proof}
By  Theorem \ref{theorem 5.25.11}
$$
u^{2^{-j}h}=
u^{(0)}+\sum_{i=1}^k\frac{h^i}{ i!2^{ji}}
u^{(i)}+ \bar r^{2^{-j}h}
 h^{k+1} ,
\quad j=0,1,...,k, 
$$
with $\bar r^{2^{-j}h}:=
 h ^{-j(k+1)}R^{2^{-j}h}$\,,
which gives 
$$
\bar u^h=\sum_{j=0}^kb_ju^{2^{-j}h}
=(\sum_{j=0}^{k}b_j)u^{(0)}
+\sum_{j=0}^{k}\sum_{i=1}^kb_j
\frac{h^i}{i!2^{ij}}u^{(i)}
+\sum_{j=0}^kb_j \bar r^{2^{-j}h}
 h^{k+1} 
$$
$$
=u^{(0)}+\sum_{i=1}^k\frac{ h^i}{i!} u^{(i)}
\sum_{j=0}^{k}\frac{b_j}{2^{ij}}
+\sum_{j=0}^kb_j \bar  r^{2^{-j}h}
=u^{(0)}+\sum_{j=0}^kb_j \bar r^{2^{-j}h}
 h^{k+1} ,
$$
since 
$$
\sum_{j=0}^{k}b_j=1,\quad  
 \sum_{j=0}^{k} b_j 2^{-ij} =0,\quad  
i=1,2,...k
$$
 by the definition of $(b_0,...,b_k)$. This and
\eqref{5.25.4} yield the result
and the theorem is proved.
 \end{proof}
 
 \begin{remark}
                                              \label{remark 1.20.10.9}
 Let the conditions of Theorem \ref{theorem 5.25.11} hold 
with  
$$
\frm=m>k+1+n+d/2, 
$$
 where $k$ and $n$ are nonnegative 
integers. Then  \eqref{2.20.10.9}  holds 
 with $\delta_{h,\lambda}\bar u^h$ and   
 $\delta_{h,\lambda}u^{(0)}$ 
 in place of  $\bar u^h$ and $u^{(0)}$, respectively, 
 for $\lambda\in \Lambda^n$.  
 \end{remark}
\begin{proof}
This follows from Theorem \ref{theorem 1.25.10.9} 
in the same way as Theorem \ref{theorem 5.25.3} 
 follows 
from Theorem \ref{theorem 5.25.11}. 
\end{proof}
 By the above remark one 
 can construct fast approximations for the derivatives 
 of $u^{(0)}$ via suitable linear combinations of finite differences 
 of $\bar u^h$.      
 
Sometimes  it suffices  to   
combine fewer terms  $u^{ 2^{-j}h }$   
to get accuracy of order $k+1$.  
For   integers 
$k\geq0$ define 
\begin{equation}                                                  \label{12.25.11.08}
\tilde{u}^{h}=\sum_{j=0}^{\tilde k}
\tilde b_ju^{2^{-j}h}\,,
\end{equation}
where  
\begin{equation*}                                                 \label{9.25.11.08}
( \tilde b_0, \tilde b_1,...,
\tilde b_{\tilde k})
:=(1,0,0,...,0)\tilde V^{-1}, 
\quad \tilde k= [\tfrac{k}{2}] , 
\end{equation*}
and  $\tilde V^{-1}$ is the inverse of the 
Vandermonde matrix with entries
$$
\tilde V^{ij}:=4^{-(i-1)(j-1)}, \quad i,j=1,...,\tilde k+1.
$$
 
\begin{theorem}
                                                                           \label{theorem 5.8.5.9}  

Let the conditions of Theorem \ref{theorem 5.25.11} hold.  
Then in the situation of Example \ref{example 5.22.2}
we have
\begin{equation}                                                          \label{8.8.5.9}
E\sup_{t\leq T} \sup_{x\in\bG_h} |\tilde u^h_{t}(x)
-u^{(0)}_{t}(x)|^{2}
\leq N |h|^{2(k+1)}\cK_m ^{2},  
\end{equation}
where  
$N$ is a constant depending only on $d$, $m$,   
$\kappa$, $K_0,\dots, K_{m+1}$, 
$ A_{0},...,A_{m} $, and $T$.  
\end{theorem} 

To prove this result we need only 
repeat the proof of Theorem \ref{theorem 5.25.3} 
taking into account that
in \eqref{1.26.4.9} we have
 $u^{(j)}_{t}=0$ for odd
$j\leq k$ since $u^{h}_{t}=u^{-h}_{t}$
owing to the fact that 
 in the case of Example \ref{example 5.22.2}  
equation \eqref{5.20.3}
does not change if we replace $h$ with~$-h$.

\begin{remark} Notice that without acceleration, i.e., when 
$k=1$ in the above theorem,  
the mean square norm of the supremum in $t$ and $x$ of the error 
of the finite difference approximations in Example 
\ref{example 5.22.2} is proportional 
to $h^2$. This is a sharp result   see, e.g., Remark 2.21 in 
\cite{DK}      
on finite difference approximations for deterministic parabolic PDEs.
\end{remark}

\begin{example}
Assume that in the situation
of  Example \ref{example 5.22.2} we have 
$d=2$ and $m=7$. Then
$$
 \tilde u^h:=\tfrac{4}{3}u^{h/2}-\tfrac{1}{3}u^h
$$ 
satisfies
$$
E\sup_{t\leq T}\sup_{x\in\bG_h}|u^{(0)}_{t}(x)
- \tilde u^h_{t}( x)|
\leq N h^{ 4}.
$$  
\end{example}

\begin{example}
Take $d=1$ and consider the following SPDE:
$$
du_{t}=3D^{2}u_{t}\,dt+2Du_{t}\,dw_{t}
$$
with initial data $u_{0}(x)=\cos x$, where $w_{t}$
is a one-dimensional Wiener process. Then a unique bounded
solution is $u_{t}(x)=e^{-t}\cos (x+2w_{t}) $. 
Example~\ref{example 5.22.2}
suggests the following version of \eqref{5.20.3}:
$$
du^{h}_{t}(x)=3\frac{u^{h}_{t}(x+h)-2u^{h}_{t}(x)+u^{h}_{t}(x-h)}
{h^{2}}\,dt+\frac{u^{h} _{t}(x+h)-u^{h}_{t}(x-h)}{h}\,dw_{t},
$$
 the  unique bounded solution 
of which with initial condition
$\cos x$ is given by
$$
u^{h}_{t}(x)=e^{-c_{h}t}\cos(x+2\phi_{h}w_{t}),\quad
h^{2}c_{h}=12\sin^{2}\frac{h}{2}-2\sin^{2}h,\quad
\phi_{h}=\frac{\sin h}{h}.
$$
For $t=1$, $h=0 .1$, and $w_{t}=0$ we have
$$
u_{1}(0)\approx0.3678794412,\quad u_{1}^{h}(0)\approx0.366352748,\quad 
u_{1}^{h/2}(0)\approx0.3674966179,
$$
$$
 \tilde u^h_{1}(0)=\tfrac{4}{3}u^{h/2}_{1}(0)
-\tfrac{1}{3}u^h_{1}(0)\approx0.3678779079.
$$ 
It is instructive to observe that such a level of
accuracy is achieved for $u^{\tilde h}_{1}(0)$ with $\tilde h=0.00316$,
which is more than 15 times smaller than $h/2$.

Actually, this example does not quite fit into our scheme
because $u_{0}$ is not square summable over $\bR$.
In connection with this we reiterate that the main
goal of the present article was to introduce a method
and not to prove the most general results. Without much
trouble our approach
can be extended to  a class of SPDEs with growing data
 by the help of weighted Sobolev spaces (see \cite{WSP}),  
and then the above example can be included formally.

\end{example}

\mysection{Auxiliary facts}
                                                \label{section 5.25.1}

The following fact is easily obtained by  Young's inequality
owing to Assumption \ref{assumption 5.19.3}.

\begin{lemma}
                                                \label{lemma 5.26.1}
 
Let Assumption \ref{assumption 5.19.3} hold. 
Then for all $\varphi\in L_2$ we have
$$
\mathbb Q_t(\varphi):
=\int_{\bR^d}\big[2\varphi(x)L^h_t\varphi(x)+
\sum_{\rho=1}^{\infty}|M_t^{h,\rho}\varphi(x)|^2\big]\,dx
$$
$$
\leq-\frac{\kappa}{2}
\sum_{\lambda\in\Lambda_{0}}\|\delta_{h,\lambda}
\varphi\|_{L_2}^2
+N\|\varphi\|^{2}_{L_2},
$$
where $N$ depends only on $\kappa$, $A_0$, $A_1$, 
and the number of 
vectors in
$\Lambda$. 
\end{lemma}
 
\begin{proof}
First observe that for $\mu\in\Lambda_0$ the 
conjugate operator
in $L_{2}$ to $\delta_{-h,\mu}$ is $
\delta_{h,\mu}$. Notice also  that 
$$
\delta_{h,\mu}(\phi\psi)=\phi\delta_{h,\mu}\psi
+(T_{h,\mu}\psi)\delta_{h,\mu}\phi,
$$
where $T_{h,\mu}\psi(x)=\psi(x+h\mu)$. 
Thus by simple calculations we get 
$$
\mathbb Q_t(\varphi)
=\sum_{i=1}^4\mathbb Q_t^{(i)}(\varphi)
$$
with 
$$
\mathbb Q_t^{(1)}(\varphi)=
- \int_{\bR^{d}}\sum_{\lambda,\mu\in\Lambda_0}
((2\ga^{\lambda\mu}_{t}-
\gb^{\lambda \rho}_{t}\gb^{\mu \rho}_{t})
(\delta_{h,\lambda}\varphi)
\delta_{h,\mu}\varphi)(x)\,dx
$$
$$
\mathbb Q_t^{(2)}(\varphi)=
2\int_{\bR^{d}}\sum_{\lambda,\mu\in\Lambda_0}
((T_{h,\mu}\varphi)(\delta_{h,\lambda}\varphi)
\delta_{h,\mu}\ga^{\lambda\mu}_{t})(x)\,dx
$$
$$
\mathbb Q_t^{(3)}(\varphi)=
\int_{\bR^{d}}2\ga^{00}_t\varphi^2(x)
+2\varphi(x)\sum_{\lambda\in\Lambda_0}
(\ga^{\lambda0}_t\delta_{h,\lambda}\varphi+
\ga^{0\lambda}_t\delta_{-h,\lambda}\varphi)(x)\,dx
$$
$$
\mathbb Q_t^{(4)}(\varphi)=
\int_{\bR^{d}}\gb^{00}_t\varphi^2(x)
+2\sum_{\lambda\in\Lambda_0}
\gb^{\lambda \rho}_{t}\gb^{0\rho}_{t}
\varphi\delta_{h,\lambda}\varphi(x)\,dx. 
$$
Due to Assumption \ref{assumption 5.19.3} (ii)
$$
\mathbb Q_t^{(1)}(\varphi)
\leq-\kappa
\sum_{\lambda\in\Lambda_0}
\|\delta_{h,\lambda}\varphi\|^2_{L_2}. 
$$
By Assumption \ref{assumption 5.19.3}(i), 
Young's inequality and the shift invariance  of 
 Lebesgue measure
$$
\mathbb Q_t^{(i)}(\varphi)\leq \frac{\kappa}{6}
\sum_{\lambda\in\Lambda_0}
\|\delta_{h,\lambda}\varphi\|^2_{L_2}+N\|\varphi\|^2, \quad i=2,3,4, 
$$
with a constant $N$ depending only on the number of elements of 
$\Lambda$, $\kappa$, $A_0$ and, for $i=2$ also on 
$A_1$. 
We finish the proof by summing up these estimates.
\end{proof}
 
 Recall the notation 
$
\bW^{m}_{2}(T)=L_{2}(\Omega
\times[0,T],\cP,W^{m}_{2}). 
$ 
Remember that $W^{m}_{2}$ denotes the Hilbert-Sobolev 
space of real-valued and also that of $l_2$-valued functions on $\bR^d$. 
Thus $\bW^{m}_{2}(T)$ denotes the Hilbert space of predictable 
functions $\phi=\phi_t$ on $\Omega\times[0,T]$ with values in the  
$W^m_2$ space of real-valued functions, 
and $\bW^{m}_{2}(T)=\bW^{m}_{2}(T,l_2)$ denotes the Hilbert space of 
functions $g=(g^{\rho})_{\rho=1}^{\infty}$ with 
values in the $W^m_2$ space of $l_2$-valued functions   
on $\bR^d$, with norm 
 defined by 
$$
\|\phi\|^2_{\bW^{m}_{2}(T)}=E\int_{0}^{T}
\|\phi_{t}\|^{2}_{W^{m}_{2}}\,dt<\infty \quad\text{and}
$$
$$
\|g\|^2_{\bW^{m}_{2}(T)}
=E\int_{0}^{T}\sum_{\rho=1}^{\infty}
\|g_{t}^{\rho}\|^{2}_{W^{m}_{2}}\,dt<\infty, 
$$
respectively. 
\begin{theorem}
                                              \label{theorem 15.25.3}
 
Let 
Assumption \ref{assumption 5.19.3} (i)  
hold.  
Let $f^{\mu}\in\bW^{\frm}_{2}(T)$, $\mu\in\Lambda$, 
 and $(g^{\rho})_{\rho=1}^{\infty}\in \bW^{\frm}_2(T)$  
be some functions.
Then  for each $h\neq 0$ 
there exists a unique continuous
$L_{2}$-valued solution $u^{h}_{t}$ of 
\begin{equation}
                                                        \label{5.26.6}
du^{h}_{t}=(\ga^{\lambda\mu}_{t}\delta_{h,\lambda}\delta_{-h,\mu}
u^{h}_{t}+\delta_{-h,\mu}f^{\mu}_{t})\,dt
+(\gb^{\lambda\rho}_{t}\delta_{h,\lambda}u^{h}_{t}+
g^{\rho}_{t})\,dw^{\rho}_{t}
\end{equation}
 for 
any $W^{\frm}_2$-valued  
$\cF_0$-measurable   
initial condition $u_{0}$. 
This solution is a
$W^{\frm}_{2}$-valued continuous process. 
 Moreover, if  
Assumption \ref{assumption 5.19.3} (ii) is also
satisfied, then 
$$
E\sup_{t\leq T}\|u^{h}_{t}\|^{2}_{W^{\frm}_{2}}
+E\int_{0}^{T}\sum_{\lambda\in\Lambda}
\|\delta_{h,\lambda}u^{h}_{t}\|^{2}_{W^{\frm}_{2}}\,dt
$$
\begin{equation}
                                                        \label{5.26.1}
\leq NE\int_{0}^{T}
\big(\sum_{\mu\in\Lambda}\|f^{\mu}_{t}\|^{2}_{W^{\frm}_{2}}
+\|g_{t}\|^{2}_{W^{\frm}_{2}}\big)\,dt
+NE \|u_{0}\|^{2}_{W^{\frm}_{2}},
\end{equation}
where $N$ depends only on $d$, ${\frm}$,     $\Lambda$,
$\kappa$, $A_0,\dots, A_{\bar{\frm}} $, and $T$. 

\end{theorem}

\begin{proof}
The first assertion is a simple consequence of the fact that
\eqref{5.20.3} is an ordinary It\^o equation 
with Lipschitz continuous coefficients for $L_{2}$-valued
processes. Similarly, \eqref{5.20.3} has a unique
$W^{\frm}_{2}$-valued solution and, 
since $W^{\frm}_{2}\subset L_{2}$,
 this proves that the $L_{2}$-valued
solution is actually $W^{\frm}_{2}$-valued. 
 Moreover, we can easily get estimate 
\eqref{5.26.1} with a constant $N$ which depends on $h$. 
In particular we have that the solution is in 
$\bW^{\frm}_{2}(T)$. 

The proof of estimate \eqref{5.26.1}
with $N$ independent of $h$ is rather standard but
still contains a point which
usually does not appear. This concerns the treatment of 
$\tilde{\ga}^{\lambda\mu}\delta_{h,\lambda}
\delta_{-h,\mu}u^{h}_{t}$ after \eqref{5.31.6} without  assuming
that 2 derivatives of $\ga$ are bounded.

By It\^o's formula for $L_{2}$-valued processes
 we find
$$
d\|u^{h}_{t}\|^{2}_{L_{2}}
=\{\mathbb Q_t(u_t^h)+2(u^h,f^{\mu}_t)
+2(b^{\lambda\rho}\delta_{h,\lambda}u_t^h,g_t^{\rho})
+\|g^{\rho}_t\|^2_{L_2}\}\,dt
$$
\begin{equation}
                                                     \label{5.31.1}
+2(u^{h}_{t},\gb^{\lambda \rho}\delta_{h,\lambda}u^{h}_{t} 
+g^{\rho}_{t})\,dw^{\rho}_{t}.
\end{equation}
 
We use Lemma \ref{lemma 5.26.1}, the inequalities like
$|ab|\leq\varepsilon a^{2}+\varepsilon^{-1}b^{2}$, and
Assumption \ref{assumption 5.19.3} (i) to conclude that
$$
E \|u^{h}_{t}\|^{2}_{L_{2}}
+ \frac{\kappa}{2} E\int_{0}^{t}
\sum_{\lambda\in\Lambda_{0}}
\|\delta_{h,\lambda}u^{h}_{s}\|^{2}_{L_{2}}\,ds
\leq
E\|u_{0}\|^{2}_{L_{2}}
$$
\begin{equation}
                                                     \label{5.31.2}
+ N
E\int_{0}^{t}\big(\|u^{h}_{s}\|^{2}_{L_{2}}+\sum_{\lambda\in\Lambda}
\|f^{\lambda}_{s}\|^{2}_{L_{2}}
+\|g_{s}\|^{2}_{L_{2}}\big)\,ds <\infty .
\end{equation}
By Gronwall's lemma we can eliminated the first term in the integral
on the right in \eqref{5.31.2} and get that
$$
E\int_{0}^{t}\sum_{\lambda\in\Lambda}
\|\delta_{h,\lambda}u^{h}_{s}\|^{2}_{L_{2}}\,ds
\leq
NE\|u_{0}\|^{2}_{L_{2}}
$$
\begin{equation}
                                                     \label{5.31.3}
+ N
E\int_{0}^{t}\big(\sum_{\lambda\in\Lambda}
\|f^{\lambda}_{s}\|^{2}_{L_{2}}
+\|g_{s}\|^{2}_{L_{2}}\big)\,ds.
\end{equation}
After that we come back to \eqref{5.31.1} and use 
Davis's  inequality to derive that
$$
E\sup_{t\leq T}\|u^{h}_{t}\|^{2}_{ L_2}
\leq NE\|u_{0}\|^{2}_{L_{2}}
$$
\begin{equation}
                                                     \label{5.31.4}
+NE\int_{0}^{T}\big(\sum_{\lambda\in\Lambda}
\|f^{\lambda}_{t}\|^{2}_{L_{2}}
+\|g_{t}\|^{2}_{L_{2}}\big)\,dt+N_{1}J,
\end{equation}
where
$$
J=E\big(\int_{0}^{T}\sum_{\rho=1}^{\infty}\big(
\int_{\bR^{d}}|u^{h}_{t}(\gb^{\lambda\rho}\delta_{h,\lambda}u^{h}_{t}
+g^{\rho}_{t})|\,dx\big)^{2}\,dt\big)^{1/2} 
$$
$$
\leq E\big(\int_{0}^{T}\|u^{h}_{t}\|_{L_{2}}^{2}
\|\gb^{\lambda}\delta_{h,\lambda}u^{h}_{t}
+g_{t}\|_{L_{2}}^{2} \,dt\big)^{1/2} 
$$
$$
\leq E\sup_{t\leq T}\|u^{h}_{t}\|_{L_{2}}
\big(\int_{0}^{T} 
\|\gb^{\lambda}\delta_{h,\lambda}u^{h}_{t}
+g_{t}\|_{L_{2}}^{2} \,dt\big)^{1/2} 
$$
$$
\leq(2N_{1})^{-1}E\sup_{t\leq T}\|u^{h}_{t}\|^{2}_{L_{2}}
+NE\int_{0}^{T}(\sum_{\lambda\in\Lambda} 
\|\delta_{h,\lambda}u^{h}_{t}\|^{2}_{L_{2}}
+\|g_{t}\|_{L_{2}}^{2})\,dt.
$$
This and \eqref{5.31.3} allow us to drop the last term
in \eqref{5.31.4} which again combined with
\eqref{5.31.3} yields
$$
E\sup_{t\leq T}\|u^{h}_{t}\|^{2}_{L_{2}}+
E\int_{0}^{T}\sum_{\lambda\in\Lambda}
\|\delta_{h,\lambda}u^{h}_{t}\|^{2}_{L_{2}}\,dt
$$
\begin{equation}
                                                           \label{5.26.4}
\leq N
E\|u_{0}\|^{2}_{L_{2}}
+ N
E\int_{0}^{T}\big(\sum_{\lambda\in\Lambda}
\|f^{\lambda}_{t}\|^{2}_{L_{2}}
+\|g_{t}\|^{2}_{L_{2}}\big)\,dt.
\end{equation}

This proves the theorem if $\frm=0$. If $\frm\geq1$,
we   differentiate \eqref{5.26.6} with respect
to $x^{i}$, and introduce the notation
$\tilde{\phi}$ for the derivative of a function
 $\phi$ in $x^{i}$. Then we obtain
\begin{equation}
                                                           \label{5.31.6}
d \tilde{u}^{h}_{t}=(\ga^{\lambda\mu}\delta_{h,\lambda}
\delta_{-h,\mu}\tilde{u}^{h}_{t}+\delta_{-h,\mu}\hat{f}^{\mu}_{t}
)\,dt+(\gb^{\lambda\rho}_{t}\delta_{h,\lambda}\tilde{u}_{t}
+\hat{g}^{\rho}_{t})\,dw^{\rho}_{t},
\end{equation}
where
$$
\hat{f}^{\mu}_{t}=\tilde{f}^{\mu}_{t},\quad\mu\ne0,\quad
\hat{f}^{0}_{t}=\tilde{f}^{0}_{t}+
\tilde{\ga}^{\lambda\mu}\delta_{h,\lambda}
\delta_{-h,\mu}u^{h}_{t},\quad \hat{g}^{\rho}_{t}=\tilde{g}^{\rho}_{t}
+\tilde{\gb}^{\lambda\rho}_{t}\delta_{h,\lambda}u^{h}_{t}.
$$
We proceed with \eqref{5.31.6} as above with
\eqref{5.26.6} with one exception that 
  for $\mu\in\Lambda_0$  we use the inequality 
(cf. Remark \ref{remark 5.29.1})
$$
E\int_{0}^{t}\int_{\bR^{d}}|\tilde{u}^{h} _{s}
\delta_{h,\lambda}
\delta_{-h,\mu}u^{h}_{s}|\,dxds
\leq   \int_0^t 
E\|\tilde{u}^{h}_{s}\|_{L_{2}}\|\delta_{h,\lambda}\partial_{\mu}
u^{h}_{s}\|_{L_{2}} \,ds 
$$
$$
\leq\varepsilon  \int_0^t 
E\|D\delta_{h,\lambda}u^{h}_{s}
\|_{L_{2}}^{2}
 \, ds 
+N\varepsilon^{-1}
 \int_0^t E\|\tilde{u}^{h}_{s}\|^{2}_{L_{2}}
 \, ds ,
$$
where $\partial_{\mu}= \mu^{i}D_{i}$ and 
$\varepsilon>0$ is arbitrary and $N$ depends only on $|\mu|$
(cf. Remark \ref{remark 5.29.1} below).
Then we come to the following counterpart of \eqref{5.31.2}
$$
E \|\tilde{u}^{h}_{t}\|^{2}_{L_{2}}
+E\int_{0}^{t}\sum_{\lambda\in\Lambda}
\|\delta_{h,\lambda}\tilde{u}^{h}_{s}\|^{2}_{L_{2}}\,ds
$$
$$
\leq
 N E\|\tilde{u}_{0}\|^{2}_{L_{2}}+(2d)^{-1}E\int_{0}^{t}
\sum_{\lambda\in\Lambda}
\|D\delta_{h,\lambda}u^{h}_{s}\|^{2}_{L_{2}}\,ds
$$
\begin{equation}
                                                           \label{5.31.7}
+ N
E\int_{0}^{t}\big(\|\tilde{u}^{h}_{s}\|^{2}_{L_{2}}
+\sum_{\lambda\in\Lambda}
\|f^{\lambda}_{s}\|^{2}_{W^{1}_{2}}
+\|g_{s}\|^{2}_{W^{1}_{2}}\big)\,ds.
\end{equation}
Recall that here $\tilde{u}^{h}_{t}$ is 
the derivative of $u^{h}_{t}$
with respect to $x^{i}$. By writing  
\eqref{5.31.7} for all $i=1,...,d$
and summing them up we see that the 
term with the factor $(2d)^{-1}$
is estimated by other terms on the 
right-hand side of \eqref{5.31.7}
and, hence, can be dropped. 
After that  the  already familiar procedure
yields
$$
E\sup_{t\leq T}\|Du^{h}_{t}\|^{2}_{L_{2}}+
E\int_{0}^{T}\sum_{\lambda\in\Lambda}
\|D\delta_{h,\lambda}u^{h}_{t}\|^{2}_{L_{2}}\,dt
$$
\begin{equation}
                                                   \label{5.26.4.9}
\leq N
E\|u_{0}\|^{2}_{W^{1}_{2}}
+ N
E\int_{0}^{T}\big(\sum_{\lambda\in\Lambda}
\|f^{\lambda}_{t}\|^{2}_{W^{1}_{2}}
+\|g_{t}\|^{2}_{W^{1}_{2}}\big)\,dt, 
\end{equation}
which along with \eqref{5.26.4} 
proves
\eqref{5.26.1} with $1$ in place of $\frm $.

Once this step is done the rest is routine.
Assume that $\frm\geq2$ and
 \eqref{5.26.1} is true with $n$ in place 
of $\frm$ for
an integer $n\in[1,\frm-1]$. 
Then we differentiate \eqref{5.26.6}
$n+1$ times and now use the notation $\tilde{\phi}$ for certain
$n+1$-th order derivative of $\phi$ with respect to $x$.
Then we will obtain \eqref{5.31.6} with slightly modified
$\hat{f}^{0}$ and $\hat{g}^{\rho}$.
 Namely, the $\hat{f}^{0}$ will be the sum of $\tilde{f}^{0}$
and the linear combination with constant
coefficients  of certain $i$-th derivatives of 
$\ga^{\lambda\mu}_{t}$ times certain $n+1-i$-th derivatives
of $\delta_{h,\lambda}\delta_{-h,\mu}u^{h}_{t}$. Here $i$
should be restricted to $[1,n+1]$. As above, the $L_{2}$-norms of the
$n+1-i$-th derivatives
of $\delta_{h,\lambda}\delta_{-h,\mu}u^{h}_{t}$ are 
dominated by the
$L_{2}$-norms of the
$n+2-i$-th derivatives
of $\delta_{ h,\lambda }u^{h}_{t}$ 
which are less than the $W^{n+1}_{2}$-norm
of $\delta_{ h,\lambda }u^{h}_{t}$ 
an estimate of which is contained
in \eqref{5.26.1}  with $n$ in place of 
$l$. Similar
changes should be made in $\hat{g}^{\rho}$.
After that we obtain the corresponding counterpart of
\eqref{5.26.4.9} which yields \eqref{5.26.1} with $n+1$
in place of $\frm$. This obviously brings the
proof of the theorem to an end.
 \end{proof}

\begin{lemma}
                                                  \label{lemma 5.27.1}
Let  $n\geq0$ be an integer, let
 $\phi\in W^{n+1}_{2}$, $\psi\in W^{n+2}_{2}$, and $\lambda,
\mu\in\Lambda_{0}$.
Set 
$$
\partial_{\lambda}\phi=\lambda^{i}D_{i}\phi,\quad
\partial_{\lambda\mu}=\partial_{\lambda}\partial_{\mu}.
$$
 Then   we have
\begin{equation}
                                                 \label{5.28.1}
\frac{\partial^{n}}{(\partial h)^{n}}\delta_{h,\lambda}\phi(x)
=\int_{0}^{1}\theta^{n} \partial_{\lambda}^{n+1}\phi
(x+h\theta\lambda)\,d\theta,
\end{equation}
$$
\frac{\partial^{n}}{(\partial h)^{n}}\delta_{h,\lambda}
\delta_{-h,\mu}\psi(x)
$$
\begin{equation}
                                                 \label{5.28.2}
=\int_{0}^{1}\int_{0}^{1}
(\theta_{1}\partial_{\lambda}-\theta_{2}
\partial_{\mu})^{n} \partial_{\lambda \mu}
\psi(x+h(\theta_{1}\lambda-\theta_{2}\mu))
\,d\theta_{1}d\theta_{2},
\end{equation}
for almost all $x\in\bR^{d}$,  for each  $h\in\bR$. 
Furthermore, if 
  $l\geq0$ is an  integer and 
$\phi\in W^{n+2+l}_{2}$ and $\psi\in W^{n+3+l}_{2}$, then
\begin{equation}
                                                 \label{5.28.3}
\big\|\delta_{h,\lambda}\phi-
\sum_{i=0}^{n}
\frac{h^{i}}{(i+1)!}\partial_{\lambda}^{i+1}
\phi\big\|_{W^{l}_{2}}\leq
\frac{|h|^{n+1}}{(n+2)!}
\|\partial_{\lambda}^{n+2}\phi\|_{W^{l}_{2}}
\end{equation}
\begin{equation}
                                                 \label{5.28.4}
\big\|\delta_{h,\lambda}
\delta_{-h,\mu}\psi-\sum_{i=0}^{n}
h^{i} \sum_{r=0}^{i}A_{i,r}
\partial_{\lambda}^{r+1}\partial_{\mu}^{i-r+1}\psi
\big\|_{W^{l}_{2}}\leq N|h|^{n+1}\|\psi\|_{W^{l+n+3}_{2}},
\end{equation}
where $N=N(|\lambda|,|\mu|,d,n)$ and
\begin{equation}                       \label{1.9.9}
A_{i,r}=\frac{(-1)^{i-r}}{(r+1)!(i-r+1)!}.
\end{equation}
\end{lemma}

\begin{proof} 
 Clearly, it suffices to prove the lemma for 
$\phi,\psi\in C_0^{\infty}(\bR^d)$.  For $n=0$ formula \eqref{5.28.1} is obtained
by applying the Newton-Leibnitz formula to $\phi(x+\theta h\lambda)$
as a function of $\theta\in[0,1]$. Applying it one more time
derives \eqref{5.28.2} from \eqref{5.28.1} for $n=0$.
After that for $n\geq1$ one obtains \eqref{5.28.1} and
 \eqref{5.28.2}  by differentiating both parts of
these equations written with $n=1$. 

Next by Taylor's formula for smooth $f(h)$ we have
$$
f(h)=\sum_{i=0}^{n}\frac{h^{i}}{i!}
\frac{d^{i}}{(dh)^{i}}f(0)+\frac{1}{n!}
\int_{0}^{h}(h-\theta)^{n}\frac{d^{n+1}}{(dh)^{n+1}}f(\theta)
\,d\theta.
$$
By applying this to
$$
\delta_{h,\lambda}\phi(x)=\int_{0}^{1}\partial_{\lambda}
\phi(x+h\theta\lambda)\,d\theta
$$
as a function of $h$ we see that
$$
\delta_{h,\lambda}\phi(x)=\sum_{i=0}^{n}
\frac{h^{i}}{(i+1)!}\partial_{\lambda}^{i+1}\phi(x)
$$
$$
+\frac{h^{n+1}}{n!}\int_{0}^{1}\int_{0}^{1}
(1-\theta_{2})^{n}\theta_{1}^{n+1}\partial_{\lambda}^{n+2}
\phi(x+h\theta_{1}\theta_{2}\lambda)\,d\theta_{1}d\theta_{2}.
$$
Now to prove \eqref{5.28.3} it only remains
to use that by Minkowski's inequality the $W^{l}_{2}$-norm
of the last term is less than the $W^{l}_{2}$-norm of
$\partial_{\lambda}^{n+2}\phi$ times
$$
\frac{|h|^{n+1}}{n!}\int_{0}^{1}\int_{0}^{1}
(1-\theta_{2})^{n}\theta_{1}^{n+1}\,d\theta_{1}d\theta_{2}
=\frac{|h|^{n+1}}{(n+2)!}.
$$
Similarly, by observing that the value at $h=0$
of the right-hand side of 
 \eqref{5.28.2} 
is
$$
n!\sum_{r=0}^{n}A_{n,r}\partial_{\lambda}^{r+1}
\partial_{\mu}^{n-r+1}\psi(x)
$$
we see that the left-hand side of \eqref{5.28.4}
is the $W^{l}_{2}$-norm of
$$
\frac{h^{n+1}}{n!}\int_{0}^{1}\int_{0}^{1}\int_{0}^{1}
(1-\theta_{3})^{n}(\theta_{1}\partial_{\lambda}-\theta_{2}
\partial_{\mu})^{n+1}\partial_{\lambda \mu}
\psi(x+h\theta_{3}(\theta_{1}\lambda-\theta_{2}\mu))\,d\theta_{1}
d\theta_{2}d\theta_{3}.
$$
This yields \eqref{5.28.4} in an obvious way.\end{proof}

\begin{remark}
                                               \label{remark 5.29.1}
Formula \eqref{5.28.1} with $n=1$ and Minkowski's
inequality imply that
$$
\|\delta_{h,\lambda}\phi\|_{L_{2}}\leq\|\partial_{\lambda}
\phi\|_{L_{2}}.
$$
By applying this inequality to finite differences of $\phi$
and using induction we easily conclude that
 $W^{l+r}_{2}\subset W^{l,r}_{h,2}$,
 where for integers $l\geq0$ and $r\geq1$
we denote by $W^{l,r}_{h,2}$ the Hilbert 
space of functions $\varphi$ on $\bR^{d}$ 
with the norm $\|\varphi\|_{l,r,h}$ defined by 
\begin{equation}                                                                                   
                                                            \label{3.6.5.9}
\|\varphi\|_{l,r,h}^2= 
\sum_{\lambda_{1},...,\lambda_{r}\in\Lambda}\|\delta_{h,\lambda_{1}}
\cdot...\cdot\delta_{h,\lambda_{r}}\varphi\|^{2}_{W^{l}_{2}}.
\end{equation}
We also set $W^{l,0}_{h,2}=W^{l}_{2}$.
Then 
for any $\phi\in W^{l+r}_{2}$ we have
$$
\|\varphi\|_{l,h,r}\leq N\|\varphi\|_{W^{l+r}_{2}},
$$
where $N$ depends only on 
$|\Lambda_0|^2:=\sum_{\lambda\in\Lambda_0}|\lambda|^2$ and 
$r$.  
\end{remark}

Set 
\begin{equation*}
\cL^{(0)}_{t}=\sum_{\lambda,\mu\in\Lambda}
\ga^{\lambda\mu}_{t}\partial_{\lambda}\partial_{\mu},
\quad
\cM^{(0)\rho}_{t}
=\sum_{\lambda\in\Lambda}\gb^{\lambda \rho}_{t}\partial_{\lambda} 
\end{equation*}
and 
for  integers $n\geq1$ introduce the  operators
$$
\cL^{(n)}_{t}=n!
\sum_{\lambda,\mu\in\Lambda_{0}}\ga^{\lambda\mu}_{t}
\sum_{r=0}^{n}A_{n,r}\partial_{\lambda}^{r+1}
\partial_{\mu}^{n-r+1}+(n+1)^{-1}
\sum_{\lambda\in\Lambda_{0}}\ga^{\lambda0}_{t}\partial_{\lambda}
^{n+1}
$$
$$
+(n+1)^{-1}\sum_{\mu\in\Lambda_{0}}\ga^{0\mu}_{t}
\partial_{\mu}^{n+1} ,
$$
$$
\cM^{(n)\rho}_{t}=(n+1)^{-1}\sum_{\lambda\in\Lambda_{0}}
\gb_{t}^{\lambda \rho}\partial_{\lambda}^{n+1}  ,
$$
$$
\cO^{h(n)}_{t}=L^{h}_{t}-\sum_{i=0}^{n}\frac{h^{i}}{i!}\cL^{(i)}_{t},
\quad
\cR^{h(n)\rho}_{t}=M^{h,\rho}_{t}-
\sum_{i=0}^{n}\frac{h^{i}}{i!}\cM^{(i)\rho}_{t}, 
$$
  where $A_{n,r}$ are defined by  \eqref{1.9.9}.  
\begin{remark}
Formally,  for $n\geq1$  the values  $\cL^{(n)}_{t}\phi$ and $\cM^{(n)\rho}_{t}\phi$ 
 are obtained
as the values at $h=0$ of the $n$-th derivatives in $h$ of $L^{h}_{t}\phi$
and $M^{h,\rho}_{t}\phi$.
\end{remark}

\begin{remark}
                                           \label{remark 5.28.01}
Owing to Assumption \ref{assumption 5.22.1}
we have 
\begin{equation}
                                               \label{5.28.5}
\cL^{(0)}_{t}=\cL_{t},\quad\cM^{(0)\rho}_{t}=\cM^{\rho}_{t}.
\end{equation}
Also observe that in light of Lemma \ref{lemma 5.27.1}, 
under Assumption \ref{assumption 5.19.3},  
for
$\phi\in W^{n+2+l}_{2}$ and $\psi\in W^{n+3+l}_{2}$ 
we have 
$$
\|\cO^{h(n)}_{t}\psi\|_{W^{l}_{2}}\leq N
|h|^{n+1} \|\psi\|_{W^{l+n+3}_{2}},
$$
\begin{equation}
                                               \label{5.28.02}
\|\cR^{h(n)}_{t}\phi\|_{W^{l}_{2}}
\leq N
|h|^{n+1} \|\phi\|_{W^{l+n+2}_{2}},
\end{equation}
where  $N$ denotes constants depending only on 
$n$, $d$, $l$, $A_0,\dots ,A_l$, and
$\Lambda$.  
\end{remark}

Let $k\in[1,m]$ be an integer. 
The functions $u^{(1)}_{t},...,u^{(k)}_{t}$  
we need in \eqref{1.26.4.9} 
will be obtained as the result of embedding of certain
functions
$v^{(i)}$ taking values in certain Sobolev spaces. 
Define $v^{(0)}_{t}$
as the solution of \eqref{3.25.4.9} from Theorem
\ref{theorem 5.25.1} and for finding 
$v^{(1)}$,...,$v^{(k)}$ introduce 
the following system  of stochastic PDEs:
\begin{align}
d v ^{(n)}_{t}=&\big(\cL_{t} v ^{(n)}_{t}+
\sum_{l=1}^{n}C_{n}^{l}\cL^{ (l)}_{t} v
^{(n-l)}_{t}\big)\,dt                       \nonumber\\
&                              
+\big(\cM^{\rho}_{t} v
^{(n)}_{t}+
\sum_{l=1}^{n}C_{n}^{l}\cM^{(l)\rho}_{t} v ^{(n-l)}_{t}
\big)\,dw^{\rho}_{t}, \quad n=1,...,k,      \label{5.28.6}
\end{align}  
 where $C^l_n=n(n-1)\cdot...\cdot(n-l+1)/l!$ 
is the binomial coefficient.  

\begin{theorem}
                                                                    \label{theorem 5.28.1}
 
Let Assumptions \ref{assumption 5.19.2},
 \ref{assumption 5.19.1}, 
  and \ref{assumption 5.19.3} (i) 
 hold, $\frm=m$, and let $1\leq k\leq m$.  
Then there exists a unique set
$ v ^{(1)}_{t},..., v ^{(k)}_{t}$
of solutions of \eqref{5.28.6} with initial condition 
$ v ^{(1)}_{0}=...= v ^{(k)}_{0}=0$ 
and such that
$ v ^{(n)}\in \bW^{m+2-n}_{2}(T)$, $n=1,...,k$. 
Furthermore, 
 with probability one  
$ v ^{(n)}_{t}$ are continuous 
$W^{m+1-n}_{2}$-valued functions
and
 there exists a constant $N$ depending only on 
$T$, $d,\kappa$,
$\Lambda$,  $m$, 
and  $K_{0}$, ..., $K_{m+1}$, 
$A_0$,....,$A_m$   such that for $n=1,...,k$
\begin{equation}
                                                  \label{5.28.7}
E\sup_{t\leq T}\| v ^{(n)}_{t}\|^{2}_{W^{m+1-n}_{2}}
+E\int_{0}^{T}\| v ^{(n)}_{t}\|^{2}_{W^{m+2-n}_{2}}\,dt
\leq N\cK_{m}^{2} .
\end{equation}
\end{theorem}
\begin{proof} Notice that for each $n= 1,\dots,k$ the  equation 
for $ v ^{(n)}_{t}$
  does not involve the unknown functions 
$ v ^{(l)}_{t}$ with indices $l>n$. Therefore we can 
prove the solvability of
\eqref{5.28.6}   and 
 the stated properties of  $ v ^{(n)}_{t}$
recursively on $n$. 

 Denote
$$
S^{(n)}=\sum_{i=1}^{n}C^{i}_{n}\cL^{(i)} v ^{(n-i)}, 
\quad 
R^{(n)\rho} =\sum_{i=1}^{n}C^{i}_{n}\cM^{(i)\rho} 
 v ^{(n-i)},
$$
and first let $n=1$.   
By Theorem \ref{theorem 5.25.1}
we have ${v}^{(0)}\in\bW^{m+2}_{2}(T)$, which 
 owing to Assumption \ref{assumption 5.19.3}(i)  yields that
$S^{(1)}\in \bW^{m-1}_{2}(T)$ and
$R^{(1)}  =(R^{(1)\rho})
\in \bW^{m}_{2}(T) $(here we need the assumption
that $\frm=m$).   
Hence, it follows again by Theorem \ref{theorem 5.25.1} that
there exists a unique $ v ^{(1)}\in\bW^{m+1}_{2}(T)$ satisfying
 \eqref{5.28.6}  with  zero initial condition.
Furthermore, $ v ^{(1)}_{t}$ is a continuous $\bW^{m}_{2}$-valued
function (a.s.) and \eqref{5.28.7} holds with $n=1$.
 
Passing to higher $n$ we assume that $m\geq k\geq2$ and    
for an  $n\in\{2,...,k \} $ 
we have found $ v ^{(1)}$,...,$ v ^{(n-1)}$ 
with the asserted properties.  Observe that for 
$i=1,\dots, n$
$$
\|\cL^{(i)}  v ^{(n-i)} 
\|_{\bW^{m-n}_{2}(T)}
\leq N\| v ^{(n-i)} \|_{\bW^{m-n+(i+2)}_{2}(T)}
$$ 
\begin{equation}
                                                  \label{5.28.8}
=
N\| v ^{(n-i)} \|_{\bW^{m+2-(n-i)}_{2}(T)}, 
\end{equation}
$$
\sum_{k=1}^{\infty}\|\cM^{(i)\rho}  v ^{(n-i)}
\|_{\bW^{m-n+1}_{2}(T)}^{2}
\leq N\| v ^{(n-i)}\|_{\bW^{m-n+1+(i+1)}_{2}(T)}^{2}
$$
\begin{equation}
                                                  \label{5.28.9}
=
N\| v ^{(n-i)} \|_{\bW^{m+2-(n-i)}_{2}(T)}^{2}. 
\end{equation}
It follows by the induction hypothesis that
$S^{(n)}\in\bW^{m-n}_{2}(T)$ 
and $R^{(n)}\in\bW^{m-n+1}_{2}(T)$. 
By applying Theorem \ref{theorem 5.25.1} we see that
there exists a unique 
$ v ^{(n)}\in\bW^{m-n+2}_{2}(T)$ 
satisfying
 \eqref{5.28.6}  with  zero initial condition. This theorem
also yields the continuity property of $ v ^{(n)}_{t}$
and an estimate, that combined with \eqref{5.28.8}
and \eqref{5.28.9} and the induction hypothesis yields 
\eqref{5.28.7}. This proves the existence. 
Uniqueness 
is obtained by inspecting the above proof in which each
$ v ^{(n)}$ was found uniquely.                                 \end{proof}

\begin{lemma}
                                            \label{lemma 5.28.5}
 Let Assumptions \ref{assumption 5.19.2}, 
 \ref{assumption 5.19.1}, 
 and \ref{assumption 5.19.3} (i) hold and $\frm=m$.  
 Let $l,k\geq0$ be integers
such that $l+k+1= m$, and let 
$ v ^{(0)},..., v ^{(k)}$ be the functions
from Theorem \ref{theorem 5.28.1}. Set
\begin{equation}                                                                   \label{1.26.10.9}
r^{h}_{t}
= v ^{h}_{t}-v^{(0)}_{t}-
\sum_{1\leq j\leq k}\frac{h^{j}}{j!}
 v ^{(j)}_{t}, 
\end{equation}
  where $v^h$ is the unique $L_2$-valued solution of 
\eqref{5.26.6} with initial condition $u_0$, $f^0=f$ 
and $f^{\mu}=0$ for $\mu\in\Lambda_0$.  
Then $r^{h}_{0}=0$, 
$r^{h}\in\bW^{m -k}_{2}(T)$, and
\begin{equation}
                                                    \label{5.29.1}
dr^{h}_{t}=(L^{h}_{t}r^{h}_{t}+F^{h}_{t})\,dt
+(M^{h, \rho }_{t}r^{h}_{t}
+G^{h, \rho }_{t})
\,dw^{ \rho}_{t},
\end{equation}
where
$$
F^{h}_{t}:=\sum_{j=0}^{k}\frac{h^{j}}{j!}\cO^{h(k-j)}_{t}
 v ^{(j)}_{t},
\quad
G^{h, \rho }_{t}:=\sum_{j=0}^{k}\frac{h^{j}}{j!}
\cR^{h(k-j)  \rho }_{t} v ^{(j)}_{t}.
$$
Finally, $F^{h}\in \bW^{l}_{2}(T)$ and 
$G^{h,\cdot}\in \bW^{l+1}_{2}(T)$.
\end{lemma}

\begin{proof} 
 Due to Assumptions 
 \ref{assumption 5.19.1} 
and \ref{assumption 5.19.3}(i)  we have 
$v^h\in \bW^{m}_{2}(T)$, and owing to Assumptions 
\ref{assumption 5.19.2}
and 
 \ref{assumption 5.19.1},
by Theorem  \ref{theorem 5.25.1} 
we have $v^{(0)}\in \bW^{m+2}(T)$. 
Hence clearly $r^h\in \bW^{m}(T)$ when $k=0$, 
and   $r^{h}\in\bW^{m -k}_{2}(T)$ follows from
Theorem \ref{theorem 5.28.1} when  $k\geq1$. 
A direct computation shows that
\eqref{5.29.1} holds with $\hat{F}$ 
and $\hat{G}$ in place of $F$
and $G$, respectively, where
$$
\hat{F}^{h}=L^{h} v ^{(0)}-\cL  v ^{(0)}
+\sum_{1\leq j\leq k}L^{h} v ^{(j)}\frac{h^j}{j!}
-\sum_{1\leq j\leq k}\cL  v ^{(j)}\frac{h^j}{j!}-I^{h}, 
$$
$$
G^{h, \rho }=M^{h, \rho } v ^{(0)}-
\cM^{ \rho} 
 v ^{(0)}
+\sum_{1\leq j\leq k}M^{h, \rho } v ^{(j)}\frac{h^j}{j!}
-\sum_{1\leq j\leq k}\cM^{ \rho}  v ^{(j)}\frac{h^j}{j!}
-J^{h, \rho} , 
$$
with
$$
I^{h}=\sum_{1\leq j\leq k}\sum_{i=1}^{j}
\frac{1}{i!(j-i)!}\cL^{(i)} v ^{(j-i)}h^j,
$$
$$
J^{h, \rho }=\sum_{1\leq j\leq k}\sum_{i=1}^{j}
\frac{1}{i!(j-i)!}\cM^{(i) \rho }  v ^{(j-i)}h^j,
$$
 where, as usual, summations over an empty set 
mean zero. 
Notice that 
$$
I^{h}=\sum_{i=1}^k\sum_{j=i}^{k}
\frac{1}{i!(j-i)!}\cL^{(i)} v ^{(j-i)}h^j 
$$
$$
=\sum_{i=1}^k\sum_{l=0}^{k-i}
\frac{1}{i!l!}\cL^{(i)} v ^{(l)}h^{l+i} 
=\sum_{l=0}^{k-1}\frac{h^{l}}{l!}
\sum_{i=1}^{k-l}\frac{h^i}{i!}\cL^{(i)} v ^{(l)}
$$
$$
=\sum_{j=0}^{k}\frac{h^{j}}{j!}
\sum_{i=1}^{ k-j} \frac{h^i}{i!}\cL^{(i)} v ^{(j)}, 
$$
and similarly, 
$$
J^{h, \rho }=\sum_{j=1}^k\sum_{i=1}^{j}
\frac{1}{i!(j-i)!}\cM^{(i) \rho } 
 v ^{(j-i)}h^j=
\sum_{j=0}^{k}\frac{h^{j}}{j!}
\sum_{i=1}^{ k-j}  \frac{h^i}{i!}\cM^{(i)
 \rho }  v ^{(j)}. 
$$
After that the fact that $\hat{F}=F$ and $\hat{G}=G$
follows by simple arithmetics. Finally, the last assertion of the lemma
immediately follows from Remark \ref{remark 5.28.01}
and Theorem \ref{theorem 5.28.1} (see however
the proof of Theorem \ref{theorem 5.29.1}).
\end{proof}
 
\mysection
{Proof of  Theorem \protect\ref{theorem 1.25.10.9} }
                                                                              \label{section proof}
In this section we suppose that $\frm=m$. 
We start with a result which, as will be seen
later, is more general than Theorem 
 \ref{theorem 1.25.10.9} .

\begin{theorem}
                                                     \label{theorem 5.29.1}
 Let 
$
m=l+k+1  
$
for some integers  $l,k\geq0$, and let 
Assumptions \ref{assumption 5.19.2}, 
\ref{assumption 5.19.1},  \ref{assumption 5.19.3}, 
and  \ref{assumption 5.22.1}  hold.  
Then for $r^{k}_{t}$, defined as in Lemma \ref{lemma 5.28.5},
 we have 
\begin{equation}
                                                   \label{5.29.2}
E\sup_{t\leq T}\|r^{h}_{t}\|^{2}_{W^{l}_{2}}
+E\int_{0}^{T}\sum_{\lambda\in\Lambda}
\|\delta_{h,\lambda}r^{h}_{t}\|^{2}_{W^{l}_{2}}\,dt
\leq N|h|^{2(k+1)}\cK_m^{2} ,
\end{equation}
where $N$ depends only on 
 $T$, $d,\kappa$,
$\Lambda$, $m$ 
and $K_{0}$,...,$K_{m+1}$, 
$A_{0}$,..., $A_{ m }$. Moreover,in the situation 
of Example \ref{example 5.22.2} we have $v^{(j)}=0$ in 
\eqref{1.26.10.9} for odd $j\leq k$.  
\end{theorem}   

\begin{proof} By Lemma \ref{lemma 5.28.5} we have
$F^{h}\in \bW^{l}_{2}(T)$ and $
G^{h,\cdot}\in \bW^{l+1}_{2}(T)$, which by 
 Lemma  \ref{lemma 5.28.5}
and Theorem \ref{theorem 15.25.3} 
yields   
that the left-hand side of
\eqref{5.29.2} is dominated by
\begin{equation}
                                                      \label{5.29.4}
 NE\int_{0}^{T}(\|F^{h}_{t}\|^{2}_{W^{l}_{2}}
+\|G^{h}_{t}\|^{2}_{W^{l}_{2}})\,dt.
\end{equation}
To estimate \eqref{5.29.4} we observe that for $j\leq k$
by Remark \ref{remark 5.28.01} we have
$$
\|\cO^{h(k-j)}_{t}u^{(j)}_{t}\|_{W^{l }_{2}}
\leq N|h|^{k-j+1}\|u^{(j)}_{t}\|_{W^{l +k-j+3}_{2}}
=N|h|^{k-j+1}\|u^{(j)}_{t}\|_{W^{m+2-j}_{2}}.
$$
Upon combining this result with Theorem \ref{theorem 5.28.1}
we see that
$$
E\int_{0}^{T}\|F^{h}_{t}\|^{2}_{W^{l}_{2}}\,dt
\leq N|h|^{2(k+1)}\cK_m^{2} .
$$
Similarly one can estimate the remaining part
of \eqref{5.29.4} thus proving 
 estimate \eqref{5.29.2}. Finally, observe 
 that in Example \ref{example 5.22.2} we have 
$v^{h}=v^{-h}$ due to the uniqueness of 
the $L_2$-valued solution for equation 
\eqref{5.20.3} with initial condition 
$u_0$. 
Hence  \eqref{5.29.2} yields $v^{(j)}=0$ for 
odd $j\leq k$. 
\end{proof}

By Sobolev's theorem on embedding of
$W^l_2$ into $C_b$  for $l>d/2$ there exists 
a linear operator $I:W^l_2\to C_b$  
 such that $I\varphi(x)=\varphi(x)$ for almost every 
$x\in\bR^d$  and 
$$
\sup_{\bR^d}|I\varphi|\leq N\|\varphi\|_{W^l_2}
$$ 
for all $\varphi\in W^l_2$, where $N$ 
is a constant depending only on $d$ and $l$. One  
has also the following  lemma on the embedding 
 $W^l_2\subset l_2(\mathbb G_h)$,  
 that we have already referred to, when we used 
Remark \ref{remark 5.25.2} on the existence 
of a unique $l_2(\mathbb G_h)$-valued continuous solution 
$\{u_{t}(x):x\in\bG_{h}\}$ to equation \eqref{5.20.3}. 

\begin{lemma}                                                 \label{lemma 2.27.4.9}
For all $\varphi\in W^{l}_2(\bR^d)$, $l>d/2$,  
$|h|\in(0,1)$
\begin{equation}                            \label{7.27.4.9}   
\sum_{x\in\mathbb G_h}|I\varphi(x)|^2|h|^d\leq N \|\varphi\|_{W^l_2}^2, 
\end{equation}
where $N$ is a constant depending only on $d$
and $l$. 
\end{lemma}

\begin{proof}
By Sobolev's embedding of $W^l_{2}$ into $C_b$, 
for $z\in\mathbb R^d$ and  smooth $\varphi$   
we have 
$$
|\varphi(z)|^2\leq \sup_{x\in B_1(0)}\varphi^2(z+hx)
\leq N \sum_{|\alpha|\leq l}h^{2|\alpha|}
\int_{B_1(0)}|(D^{\alpha}\varphi)(z+hx)|^2\,dx
$$
$$
=N \sum_{|\alpha|\leq l}|h|^{2|\alpha|-d}
\int_{B_h(z)}|(D^{\alpha}\varphi)(x)|^2\,dx
\leq N|h|^{-d}\sum_{|\alpha|\leq l}
\int_{B_h(z)}|(D^{\alpha}\varphi)(x)|^2\,dx, 
$$
with a constant  $N=N(d,l)$, where  $B_r(z)=\{x\in\bR^d:|x-z|<r\}$.   
Thus 
$$
\sum_{z\in\mathbb G_h}|\varphi(z)|^2|h|^d
\leq N
\sum_{|\alpha|\leq l}\sum_{z\in\mathbb G_h}
\int_{B_h(z)}|(D^{\alpha}\varphi)(x)|^2\,dx
$$
that yields \eqref{7.27.4.9}.  
\end{proof}

Set $R^{h}_t=Ir^{h}_t$.  
Recall that  $\Lambda^0=\{0\}$, $\delta_{h,0}$ is the identity operator and 
$
\delta_{h,\lambda}=\delta_{h,\lambda_1}
\cdot...\cdot \delta_{h,\lambda_n} 
$
for $(\lambda_1,\dots,\lambda_n)\in \Lambda^n$, $n\geq1$.   
Then we have the following corollary of 
Theorem \ref{theorem 5.29.1}

\begin{corollary}                                                            \label{corollary 1.12.10.9}
Let the assumptions of Theorem \ref{theorem 5.29.1} 
hold with $l>n+d/2$ for some integer $n\geq0$. Then 
for $\lambda\in\Lambda^n$ we 
have 
$$
E\sup_{t\in[0,T]}\sup_{x\in\bR^d}
|\delta_{h,\lambda} R^h_{t}( x)|^2
\leq 
Nh^{2(k+1)}\cK_m^2
$$
$$
E\sup_{t\in[0,T]}\sum_{x\in\mathbb G_h}
|\delta_{h,\lambda}R^h_{t}( x)|^2|h|^d
\leq 
Nh^{2(k+1)}\cK_m^2
$$  
with a constant $N$ depending only on 
$\Lambda$, $d$, $m$, $K_0$, ..., $K_{m+1}$, 
$A_{0}$,..., $A_{ m }$, 
$\kappa$, and $T$. 
\end{corollary}
 
\begin{proof} 
Set $j=n-l$. Then $j>d/2$ and  using  
Sobolev's theorem on embedding 
$W^{j}_2$ into $C_b$ and  taking into account  
Remark  \ref{remark 5.29.1}, 
from Theorem \ref{theorem 5.29.1} we get 
$$
E\sup_{t\in[0,T]}\sup_{x\in\bR^d}|\delta_{h,\lambda}R^h_{t}( x)|^2
\leq C_1E\sup_{t\in[0,T]}\|R^h_{t}\|_{j,h,n}^2
$$
$$
\leq C_2E\sup_{t\in[0,T]}\|R^h_{t}\|_{W^l_2}^2
\leq 
Nh^{2(k+1)}\cK_m^2,
$$
 where $C_1$ and $C_2$ are constants 
depending only on $m$ and $d$, and 
$N$  is a constant  depending only on $m$, $d$, $T$, $\kappa$, 
$\Lambda$ and $K_i$ for 
$i\leq m+1$. 
Similarly, by Lemma \ref{lemma 2.27.4.9} and 
Remark  \ref{remark 5.29.1}
$$
E\sup_{t\in[0,T]}
\sum_{x\in\mathbb G_h}|\delta_{h,\lambda}R^h_{t}( x)|^2|h|^d
\leq C_1E\sup_{t\in[0,T]}\|\delta_{h,\lambda}R^h_{t}\|_{W_2^j}^2
$$
$$
\leq C_2E\sup_{t\in[0,T]}\|R^h_{t}\|_{W^l_2}^2
\leq 
Nh^{2(k+1)}\cK_m^2. 
$$
\end{proof}

Now we show that 
Theorem \ref{theorem 1.25.10.9} follows 
from the above corollary.  We define 
$$
\hat u^h=Iv^h,\quad u^{(j)}=Iv^{(j)}, \quad j=0,...,k, 
$$
where $v^h$ is the unique $\cF_t$-adapted 
continuous $L_2(\bR^d)$-valued solution of equation \eqref{5.20.3} 
with initial condition $u_0$, the processes 
$v^{(0)}$,...,$v^{(k)}$ are given by Theorem 
\ref{theorem 5.28.1},  and
$I$ is the embedding operator from $W^l_{2}$ into $C_b$. 
By virtue of Theorem \ref{theorem 15.25.3}, 
$v^h$ is a continuous 
$W^l_2$-valued process, and 
by 
Theorem \ref{theorem 5.28.1}
$v^{(j)}$, $j=1,2,...,k$,  are 
$W^{n+1-k}_2$-valued continuous processes. 
Since 
$l>d/2$ and $n+1-k>d/2$, the processes 
$\hat u^{h}$ and $u^{(j)}$ are 
well-defined and clearly 
\eqref{1.26.10.9} implies \eqref{1.26.4.9}
with $\hat{u}^{h}$ in place of $u^{h}$.  
To show that Corollary \ref{corollary 1.12.10.9}
yields Theorem \ref{theorem 1.25.10.9} 
 we need only show that almost surely 
\begin{equation}                                                                      \label{1.22.10.9}
\hat u^{h}_t(x)=u_t^h(x)
\quad \text{for all $t\in[0,T]$} 
\end{equation}
for each $x\in \mathbb G_h$, where $u^h$ is the unique 
$\cF_t$-adapted $l_2$-valued continuous solution 
of \eqref{5.20.3}. 
To see this let $\varphi$ be a compactly 
supported nonnegative smooth function on $\mathbb R^d$ with unit integral, 
and for a fixed $x\in \mathbb G_h$ set 
\begin{equation*}                                         
\varphi_{\varepsilon}(y)
=\varphi((y-x)/\varepsilon)
\end{equation*}
for $y\in\bR^d$ and $\varepsilon>0$. 
Since $\hat u^h$ is a continuous $L_2$-valued solution of 
\eqref{5.20.3}, for each $\varepsilon$ almost surely 
$$
\int_{\bR^d}\hat u^h_t(y)\varphi_{\varepsilon}(y)\,dy
=\int_{\bR^d}\hat u(y)\varphi_{\varepsilon}(y)\,dy
+\int_0^t\int_{\bR^d}(L^{h}_{s}\hat u^{h}_{s}(y)
+f_{s}(y))\varphi_{\varepsilon}(y)\,dy\,ds
$$
$$
+\int_0^t\int_{\bR^d}
(M^{h,\rho}_{s}\hat u^{h}_{s}(y)
+g^{\rho}_{s}(y))\varphi_{\varepsilon}(y)\,dy\,dw^{\rho}_s
$$
 for all $t\in[0,T]$. 
Letting here $\varepsilon\to0$ we see 
 that both sides 
 converge  in probability, 
 uniformly in $t\in[0,T]$, and thus we 
 get that almost surely 
 $$
\hat u^h_t(x)
= u_0(x)
+\int_0^t\big[L^{h}_{s}\hat u^{h}_{s}(x)
+f_{s}(x)\big]\,ds+
\int_0^t\big[
M^{h,\rho}_{s}\hat u^{h}_{s}(x)
+g^{\rho}_{s}(x)\big]\,dw^{\rho}_s
$$
 for all $t\in[0,T]$. (Remember that $u_0$, $f$ 
 and $g$ are continuous in $x$ 
 by virtue of Remark \ref{remark 1.27.10.9}.) 
  Moreover, owing to Lemma \ref{lemma 2.27.4.9} 
 the restriction of $\hat u_t$ onto $\mathbb G_h$ is a continuous 
 $l_2(G_h)$-valued process. Hence, because of the uniqueness 
 of the $l_2(\mathbb G_h)$-valued continuous $\cF_t$-adapted solution of 
 \eqref{5.20.3} for any $l_2$-valued $\cF_0$-measurable 
 initial condition, we have \eqref{1.22.10.9}, that finishes the proof 
 Theorem~\ref{theorem 1.25.10.9}.

 Theorem \ref{theorem 1.25.10.9} yields the following 
generalisation of Theorem  \ref{theorem 5.25.3}.

\begin{theorem}                  
Let the conditions of Theorem \ref{theorem 1.25.10.9} hold
with $n=0$. 
Then 
$$
E\sup_{t\leq T}\sup_{x\in\bG_h}|\bar u^h_{t}( x)-
u^{(0)}_{t}( x)|^{2}
$$
\begin{equation}                                     \label{1.25.10.9}
+
E\sup_{t\leq T}\sum_{x\in\bG_h}|\bar u^h_{t}( x)-
u^{(0)}_{t}( x)|^{2}|h|^{d}
\leq N |h|^{2(k+1)}\cK_m^{2} , 
\end{equation} 
where $\bar u^h$ is defined by \eqref{12.25.11.08} and  
$N$ depends only on $\Lambda$, $d$, $m$, $K_0$, ..., $K_{m+1}$, 
$A_{1}$,...,$A_{m}$, 
$\kappa$, and $T$. In the situation of Example \ref{example 5.22.2} 
estimate \eqref{1.25.10.9} holds also for $\tilde u^h$, defined by 
\eqref{12.25.11.08}, in place of $\bar u$. 
\end{theorem}

\end{document}